\documentclass[11pt,reqno]{amsart}
\usepackage{graphicx}
\usepackage{amsmath,amssymb,amsthm}
\usepackage[mathlines]{lineno}
\usepackage{hyperref}
\usepackage{mathtools}
\usepackage{enumitem}
\usepackage{xcolor}
\usepackage[margin=1in]{geometry}

\theoremstyle{plain}
\newtheorem{theorem}{Theorem}[section] 
\newtheorem{claim}[theorem]{Claim}
\newtheorem{lemma}[theorem]{Lemma}
\newtheorem{corollary}[theorem]{Corollary}
\newtheorem{definition}[theorem]{Definition}

\newcommand{\cH}{\mathcal{H}}

\begin{document}

 \title{Independent transversals in hypergraphs with low codegrees}
\begin{abstract}
We prove that for every fixed $r\geq3$, $k>0$ and $\varepsilon\in(0,1)$, there is $\gamma>0$ such that the following holds for all sufficiently large $b$.
An $r$-uniform hypergraph $\cH$ whose vertex set is partitioned into parts of size at least $b$ contains at least $\lceil(1-\varepsilon)b\rceil$ pairwise vertex-disjoint independent transversals, provided that the average degree in each part is at most $kb^{r-1}$ and maximum $j$-codegree $\Delta_j(\cH)\leq\gamma b^{r-j}$ for $2\leq j<r$.
Under the same average degree, if $q$ is the minimum part size and $\Delta_j(\cH)=o(q^{r-j})$ for $2\leq j<r$, the maximum number of pairwise vertex-disjoint independent transversals is $(1-o(1))q$ as $q\to\infty$.
We develop an iterative halving reduction for partitioned hypergraphs to apply the theorem of Delcourt and Postle on matchings avoiding forbidden submatchings.
\end{abstract}

\author{Hanle Dai}
\address{School of Mathematical Sciences, Beijing University of Posts and Telecommunications, Beijing, China}
\email{hanledai@bupt.edu.cn}
\author{Luyining Gan}
\address{School of Mathematical Sciences, Beijing University of Posts and Telecommunications, Beijing, China; Key Laboratory of Mathematics and Information Networks (BUPT), Ministry of Education, China}
\email{elainegan@bupt.edu.cn}
\thanks{L. Gan was partially supported by the National Natural Science Foundation of China (12401446) and the Fundamental Research Funds for the Central Universities.}
\maketitle

\section{Introduction} 
Let $\cH=(V, E)$ be an $r$-uniform hypergraph (or briefly an $r$-graph) and let $\mathcal P=\{V_1,\ldots,V_m\}$ be a partition of $V(\cH)$. 
A \emph{transversal} of $\mathcal P$ is a set $T\subseteq V(\cH)$ such that $|T\cap V_i|=1$ for every $i\in[m]$~\footnote{Denote by $[k]$ the set of integers from $1$ to $k$.}.
A \emph{independent transversal} of $\cH$ with respect to $\mathcal P$ is a transversal $T$ such that $e\nsubseteq T$ for all $e\in E(\cH)$.
In the graph case $r=2$, the search for sufficient conditions guaranteeing independent transversals in graphs dates back to the early 1970s by Erd\H{o}s~\cite{Er72} and Bollob\'as, Erd\H{o}s and Szemer\'edi~\cite{BES75}. Since then, the problem has attracted considerable attention, both for its intrinsic combinatorial interest~\cite{GS22, LS07, ST06, Yu97a, Yu97b} and for its connections to linear arboricity, list colouring, and strong colouring~\cite{Al88, Ha01, Ha04, Al92, GS22}.

The notion extends naturally to hypergraphs.
For $r\ge3$, the corresponding existence problem has also been studied under sparsity and degree conditions. Erd\H{o}s, Gy\'arf\'as and {\L}uczak~\cite{EGL94} studied independent transversals in sparse partite $r$-graphs with parts of the same size and exactly one edge across any $r$ distinct parts, and they obtained bounds on the maximum number of parts guaranteeing the existence of an independent transversal. 
For an $r$-graph $\cH$ and a nonempty set $U\subseteq V(\cH)$, write $\overline d_{\cH}(U):=|U|^{-1}\sum_{v\in U}d_{\cH}(v)$ for the average degree of $U$ in $\cH$, where $d_{\cH}(v)$ is the number of edges of $\cH$ containing $v$.
Wanless and Wood~\cite{WW22} proved that an $r$-graph $\cH$ with a partition $V(\cH)=V_1\cup\cdots\cup V_m$ has at least $\left((r-1)b/r\right)^m$ distinct independent transversals, provided that $|V_i|\geq b$ and $\overline d_{\cH}(V_i)\leq\alpha_r b^{r-1}$ for every $i\in[m]$, where $\alpha_r:=(r-1)^{r-1}/r^r$.
Groenland, Kaiser, Treffers and Wales~\cite{GKTW23} proved that the coefficient $\alpha_r$ is asymptotically best possible by constructing, for every $\delta>0$ and every sufficiently large integer $b$, an $r$-graph with $|V_i|=b$ and $\overline d_{\cH}(V_i)\leq(\alpha_r+\delta)b^{r-1}$, yet the hypergraph has no independent transversal.
This motivates the question of whether additional codegree bound allows $\alpha_r$ to be replaced by any larger average degrees to guarantee an independent transversal.
 
\subsection{Our result} Let $\cH$ be an $r$-graph $\cH$ and let $S\in {V \choose j}$ with $j \in [r]$.
 The \emph{$j$-codegree} of $S$ in $\cH$, denoted $d_{\cH}(S)$, is the number of $(r-j)$-subsets $Y$ such that $S\cup Y$ is an edge in $\cH$.
 The \emph{maximum $j$-codegree} of $\cH$, denoted $\Delta_j(\cH)$, is the maximum of $d_{\cH}(S)$ over all $S \in {V\choose j}$. 
Since the coefficient $a_r=(r-1)^{r-1}/r^r$ in the average degree bound given by Wanless and Wood~\cite{WW22} is fixed and relatively small, it is natural to ask whether such a coefficient can be any fixed positive constant provided that the maximum $j$-codegrees are sufficiently small for $j\geq2$. 
Our main result gives an affirmative answer, and moreover guarantees linearly many pairwise vertex-disjoint independent transversals.
\begin{theorem}\label{thm:mainthm}
For every integer $r\geq 3$, any $k>0$, and $\varepsilon\in(0,1)$, there exist
$\gamma>0$ and $b_0>0$ such that the following holds for each $b\geq b_0$ and integer $m\geq 1$.
Let $\cH$ be an $r$-uniform hypergraph with a partition
$V(\cH)=V_1\cup\cdots\cup V_m$. Suppose that
\begin{itemize}
    \item $|V_i|\geq b$ for every $i\in[m]$,
    \item $\overline d_{\cH}(V_i)\leq kb^{r-1}$ for every $i\in[m]$,
    \item $\Delta_j(\cH)\leq \gamma b^{r-j}$ for every $2\leq j<r$.
\end{itemize}
Then $\cH$ contains at least $\lceil(1-\varepsilon)b\rceil$ pairwise vertex-disjoint independent transversals.
\end{theorem}
Note that the constants $\gamma:=\gamma(r, k, \varepsilon)$ and $b_0:=b_0(r, k, \varepsilon)$ are independent of the number of parts.
Moreover, the theorem allows vertex degrees to be substantially larger than the average degree bound.

The following corollary of Theorem~\ref{thm:mainthm}, stated in terms of a common upper bound on the average degree in each part, shows that the maximum number of pairwise vertex-disjoint independent transversals is asymptotic to the minimum part size.

\begin{corollary}\label{cor:asymptotic-packing}
Fix an integer $r\geq3$ and a constant $a>0$.
Let $(\cH_n)_{n\geq1}$ be a sequence of $r$-graphs, where each $\cH_n$ has a vertex partition $\mathcal P_n=\{V_{n,1},\ldots,V_{n,m_n}\}$.
Write $q_n=\min_{i\in[m_n]}|V_{n,i}|$, and let $(D_n)_{n\geq1}$ be a sequence of positive real numbers tending to infinity.
Suppose that for every $n\geq1$ and every $i\in[m_n]$
\begin{itemize}
    \item $|V_{n,i}|\geq aD_n^{1/(r-1)}$,
    \medskip
    \item $\overline d_{\cH_n}(V_{n,i})\leq D_n$,
    \medskip
    \item $\Delta_j(\cH_n)/D_n^{(r-j)/(r-1)}\to0$ as $n\to\infty$ for every $2\leq j<r$.
\end{itemize}
Then the maximum number of pairwise vertex-disjoint independent transversals of $\cH_n$ with respect to $\mathcal P_n$ is $(1-o(1))q_n$ as $n\to\infty$.
\end{corollary}

Indeed, for any fixed $\varepsilon\in(0,1)$ and all sufficiently large $n$, apply Theorem~\ref{thm:mainthm} with $b=q_n$ and $k=a^{-(r-1)}$.
Since $\varepsilon$ is arbitrary and $q_n$ is an upper bound, this gives the stated asymptotic estimate.

\subsection{Proof ideas} 

To find a single independent transversal, we first discard vertices of large degree and keep equal-size subsets of the remaining parts, which converts the average degree assumption into a maximum codegree bound.
The problem can then be expressed in terms of matchings avoiding forbidden submatchings in an auxiliary bipartite graph.
Delcourt and Postle's theorem~\cite{DP2026} applies in this setting, but requires a polynomial saving in the codegrees.
To overcome this restriction, we develop a halving reduction for multipartite hypergraphs by combining part size control idea from Loh and Sudakov~\cite{LS07} with codegrees control idea from Cooper and Mubayi~\cite{CM16}.
 It allows the polynomial saving to be replaced by a sufficiently small fixed coefficient in the original codegree bounds. We repeatedly retain approximately half of each part, preserving equal part sizes and controlling the maximum codegrees of all relevant sizes. The iteration stops at a sufficiently large constant part size. Thus, the halving process converts the sufficiently small fixed coefficient in the original maximum codegree bounds into the polynomial saving required by the Delcourt and Postle's theorem. Repeating the entire argument after deleting each transversal yields at least $\lceil(1-\varepsilon)b\rceil$ pairwise vertex-disjoint independent transversals.

\medskip

The remainder of the paper is organized as follows.
In Section~\ref{sec:pre}, we introduce configuration hypergraphs, collect the probabilistic tools and useful lemmas.
We establish the halving lemma and derive the criterion for the existence of independent transversals in Section~\ref{sec:halving} and prove the main theorem, Theorem~\ref{thm:mainthm}, in Section~\ref{sec:pf-main}.

\section{Preliminaries}\label{sec:pre}

We recall the notion of a configuration hypergraph~\cite[Definition~1.6]{DP2026}, restricting to the setting needed in our proof.

\begin{definition}[Configuration hypergraph]
Let $r\geq3$ be an integer and let $G=(A,B)$ be a bipartite graph.
An $r$-graph $\mathcal F$ is a \emph{configuration hypergraph} for $G$ if $V(\mathcal F)=E(G)$ and $E(\mathcal F)$ consists of a set of matchings of size $r$ in $G$.
\end{definition}
A matching $M\subseteq E(G)$ is \emph{$\mathcal F$-avoiding}~\footnote{Such a matching is also called \emph{conflict-free} in the literature, see~\cite{GJKKL24}.} if $F\nsubseteq M$ for every $F\in E(\mathcal F)$.
A matching of $G$ is \emph{$A$-perfect} if it covers every vertex of $A$.
The independent transversal problem can be reformulated as the problem of finding an $\mathcal F$-avoiding $A$-perfect matching in a suitable bipartite graph.
We use the following special case of a theorem of Delcourt and Postle, stated in our notation, to derive the sufficient condition for independent transversals in Lemma~\ref{lem:indep-trans}.
\begin{lemma}~\cite[Theorem~1.16]{DP2026}\label{lem:DP}
For every integer $r\geq3$ and every $\beta\in(0,1)$, there exist an integer $D_\beta\geq 0$ and a constant $\alpha>0$ such that the following holds for every $D\geq D_\beta$.
Let $G=(A,B)$ be a bipartite graph such that every vertex in $A$ has degree at least $(1+D^{-\alpha})D$ and every vertex in $B$ has degree at most $D$.
Let $\mathcal F$ be an $r$-uniform configuration hypergraph for $G$.
Suppose that $\Delta_1(\mathcal F)\leq\alpha D^{r-1}\log D$ and $\Delta_\ell(\mathcal F)\leq D^{r-\ell-\beta}$ for every $2\leq\ell<r$.
Then $G$ contains an $\mathcal F$-avoiding $A$-perfect matching.
\end{lemma}

We next collect the probabilistic tools used in the proof of the halving lemma in Section~\ref{sec:pf-main}, including two concentration inequalities and the Lov\'asz Local Lemma. Concentration inequalities are the tools for controlling the deviation of random variables from their expectations.

We first state Hoeffding's inequality~\cite[Theorem~2]{H63}, which shows that the sum of independent bounded random variables is sharply concentrated around its expectation.
It will be used to control the sizes of the retained parts.
\begin{lemma}[Hoeffding's inequality~\cite{H63}]\label{lem:H-ineq}
Let $X_1,\ldots,X_n$ be independent random variables satisfying $a_i\leq X_i\leq b_i$ almost surely for every $i\in[n]$, and set $X:=\sum_{i=1}^n X_i$.
Then, for every $t>0$,
\[
\Pr\left(\mathbb E [ X]-X\geq t\right)\leq\exp\left(-\frac{2t^2}{\sum_{i=1}^n(b_i-a_i)^2}\right).
\]
\end{lemma}

To control the codegree bounds, the second concentration inequality we will use is the bounded differences inequality, also known as McDiarmid's inequality~\cite{M89}. 
It provides concentration for a function of independent random variables when changing any single variable changes the value of the function by only a bounded amount.

\begin{lemma}[Bounded differences inequality~\cite{M89}]\label{lem:Bd-ineq}
Let $X_1,\ldots,X_n$ be independent random variables, where $X_i$ takes values in a set $\Omega_i$.
Let $f:\Omega_1\times\cdots\times\Omega_n\to\mathbb R$ be a measurable function.
Suppose that there are nonnegative constants $c_1,\ldots,c_n$ such that, for every $i\in[n]$, we have $|f(x)-f(x')|\leq c_i$ whenever $x,x'\in\Omega_1\times\cdots\times\Omega_n$ differ only in their $i$th coordinate.
Set $Z:=f(X_1,\ldots,X_n)$.
Then, for every $t>0$,
\[
\Pr\left(Z-\mathbb E[ Z]\geq t\right)\leq\exp\left(-\frac{2t^2}{\sum_{i=1}^n c_i^2}\right).
\]
\end{lemma}

The concentration inequalities above show that the probability that a random variable deviates substantially from its expectation is small. 
We refer to such undesirable deviations as bad events, that is, the failures of the required size and codegree bounds. 
To ensure that all these bad events do not occur, we use the following Lov\'asz local lemma.
\begin{lemma}[Lov\'asz local lemma {\cite{AS16}}]\label{lem:lll}
Let $\mathcal{A}$ be a finite collection of events such that $\Pr(A)\leq p$ for every $A\in\mathcal{A}$.
Suppose that each $A\in\mathcal{A}$ is mutually independent of all but at most $d$ other events in $\mathcal{A}$.
If $ep(d+1)\leq 1$, then
\[
\Pr\left(\bigcap_{A\in\mathcal{A}}\overline{A}\right)>0.
\]
\end{lemma}

\section{Halving lemma and Independent transversal criterion}\label{sec:halving}
In this section, we prove two fundamental lemmas. Lemma~\ref{lem:indep-trans} gives a sufficient condition for the existence of an independent transversal, while Lemma~\ref{lem:halving} allows us to reduce the conditions of our main theorem to those of Lemma~\ref{lem:indep-trans}.

The following halving lemma adapts the earlier random halving arguments~\cite{LS07,CM16} to a given vertex partition.
It allows us to keep approximately half of each part while controlling the maximum codegrees.

\begin{lemma}[Halving lemma]\label{lem:halving}
Fix $r\geq 3$ and $\lambda\geq 1$. There exist constants $\ell, z, q_0 > 0$ such that the following holds. Let $\cH$ be an $r$-graph and $V(\cH)=V_1\cup\cdots\cup V_m$ be a partition of $\cH$. Suppose $q:=|V_i|\ge q_*$
 and $\Delta_j(\cH)\leq \lambda q^{r-j}$ for every $j\in[r-1]$. Then for $i\in [m]$ and $j\in[r-1]$ there is an induced subhypergraph $H'$ such that 
     \[q'\coloneq |V_i\cap V(H')| = \left\lfloor \frac{q}{2}-\ell\sqrt{q\log q}\right\rfloor,\]
     and 
     \begin{equation}
     \label{equ:norm-degree}
     \frac{\Delta_j(H')}{(q')^{r-j}}\leq\frac{\Delta_j(\cH)}{q^{r-j}}+z\sigma(q),
     \end{equation}
    where $\sigma(q)=\sqrt{\frac{\log q}{q}}$.
\end{lemma}

\begin{proof}
Choose $\ell=\ell(r,\lambda)\geq1$ such that $2\ell^2/((r-1)\lambda^2)\geq2r+2$, and let $q^*=q^*(r,\lambda)$ be sufficiently large.
Recall that $\sigma(q)=\sqrt{(\log q)/q}$.

Select each vertex of $V(\cH)$ independently with probability $1/2$, let $R$ be the resulting set, and write $I_x=\mathbf{1}_{\{x\in R\}}$ for each $x\in V(\cH)$.
For each $i\in[m]$, set $Y_i=|V_i\cap R|$ and define the bad event $A_i:=\{Y_i<q/2-\ell\sqrt{q\log q}\}$.
Since $Y_i\sim\operatorname{Bin}(q,1/2)$ and $\mathbb E [Y_i]=q/2$, Lemma~\ref{lem:H-ineq} gives
\[
\Pr(A_i)\leq\exp(-2\ell^2\log q)=q^{-2\ell^2}.
\]

Next, fix $j\in[r-1]$ and a set $S\in\binom{V(\cH)}{j}$ with $d_{\cH}(S)>0$, and put $s=r-j$.
Let $X_S$ denote the number of the edges $e\in E(\cH)$ such that $S\subseteq e$ and $e\setminus S\subseteq R$.
By linearity of expectation, $\mathbb E [X_S]=2^{-s}d_{\cH}(S)$.
Define the bad event $B_S:=\{X_S>2^{-s}d_{\cH}(S)+\ell q^{s-1/2}\sqrt{\log q}\}$.

The variable $X_S$ is a function of the independent variables $\{I_x:x\in V(\cH)\setminus S\}$.
Changing only $I_x$ changes $X_S$ by at most $c_x:=d_{\cH}(S\cup\{x\})$.
Moreover, $\sum_{x\notin S}c_x=sd_{\cH}(S)$, since each edge containing $S$ contributes exactly $s$ to this sum.
Consequently,
\[
\sum_{x\notin S}c_x^2\leq\Delta_{j+1}(\cH)\sum_{x\notin S}c_x=s\Delta_{j+1}(\cH)d_{\cH}(S)\leq s\lambda^2q^{2s-1}.
\]
Here the bound $\Delta_{j+1}(\cH)\leq\lambda q^{s-1}$ also holds when $j=r-1$, since $\Delta_r(\cH)\leq1\leq\lambda$.
Applying Lemma~\ref{lem:Bd-ineq} with $t=\ell q^{s-1/2}\sqrt{\log q}$ therefore gives
\[
\Pr(B_S)\leq\exp\left(-\frac{2\ell^2\log q}{s\lambda^2}\right)=q^{-2\ell^2/(s\lambda^2)}.
\]
By our choice of $\ell$, every bad event has probability at most $p:=q^{-2r-2}$.

We next bound the dependencies among the bad events.
Each $A_i$ is determined by exactly $q$ random variables $\{I_x:x\in V_i\}$.
Each $B_S$ is determined by the variables $I_x$ for which $x\notin S$ and $d_{\cH}(S\cup\{x\})>0$, of which there are at most $sd_{\cH}(S)\leq(r-1)\lambda q^{r-1}$.
Thus every bad event is determined by at most $(r-1)\lambda q^{r-1}$ variables.
For a fixed vertex $x\in V(\cH)$, the variable $I_x$ appears in exactly one bad event $A_i$.
It can appear in $B_S$ only if nonempty $S\subseteq e\setminus\{x\}$ for some $e$ containing $x$.
There are $2^{r-1}-1$ such subsets for each edge $e$, so $I_x$ appears in at most
$1+(2^{r-1}-1)d_{\cH}(x)\leq2^{r-1}\lambda q^{r-1}$
bad events.

Join two bad events if they depend on a common variable $I_x$.
Since the variables $I_x$ are independent, this gives a dependency graph with maximum degree $d$ satisfying $d+1\leq(r-1)2^{r-1}\lambda^2q^{2r-2}$.
Hence
$ep(d+1)\leq e(r-1)2^{r-1}\lambda^2q^{-4}<1$
for all $q\geq q^*$, provided that $q^*$ is sufficiently large.
By Lemma~\ref{lem:lll}, with positive probability none of the bad events occurs.
Fix such a choice of $R$.

For every $i\in[m]$, we have $|V_i\cap R|\geq q/2-\ell\sqrt{q\log q}$.
We may therefore choose a set $W_i\subseteq V_i\cap R$ of size $q':=\lfloor q/2-\ell\sqrt{q\log q}\rfloor$.
Let $\cH':=\cH[\bigcup_{i\in[m]}W_i]$.
Fix $j\in[r-1]$, put $s=r-j$, and consider a set $S\in\binom{V(\cH')}{j}$ with $d_{\cH}(S)>0$.
Every edge of $\cH'$ containing $S$ is counted by $X_S$, and the event $B_S$ does not occur.
Thus
\[
d_{\cH'}(S)\leq X_S\leq2^{-s}d_{\cH}(S)+\ell q^s\sigma(q).
\]
Taking the maximum over all $j$-sets gives
\[
\Delta_j(\cH')\leq2^{-s}\Delta_j(\cH)+\ell q^s\sigma(q).
\]

It remains to account for the change in part size.
Increasing $q^*$ if necessary, we may assume that $q\geq e$ and $8(r-1)\ell\sigma(q)\leq1$.
Since $\ell\geq1$ and $1/q\leq\sigma(q)$, we have
$q'\geq\frac q2-\ell q\sigma(q)-1\geq\frac q2\left(1-4\ell\sigma(q)\right)$.
For every $s\in[r-1]$, Bernoulli's inequality implies
$\left(1-4\ell\sigma(q)\right)^{-s}\leq\left(1-4s\ell\sigma(q)\right)^{-1}\leq1+8s\ell\sigma(q)\leq2$.
Combining these estimates and using $\Delta_j(\cH)/q^s\leq\lambda$, we obtain
\[
\frac{\Delta_j(\cH')}{(q')^s}\leq\left(\frac{\Delta_j(\cH)}{q^s}+2^s\ell\sigma(q)\right)\left(1-4\ell\sigma(q)\right)^{-s}\leq\frac{\Delta_j(\cH)}{q^s}+\ell\left(8s\lambda+2^{s+1}\right)\sigma(q).
\]
Since $s=r-j$, the desired conclusion follows by taking $z:=\ell(8(r-1)\lambda+2^r)$.
\end{proof}

 Before stating the next lemma, we make a simple observation about the vertex partition.
For an $r$-graph $\cH$ with a partition $\mathcal P=\{V_1,\ldots,V_m\}$, an edge $e\in E(\cH)$ is \emph{crossing} with respect to $\mathcal P$ if $|e\cap V_i|\leq1$ for every $i\in[m]$.
Since a non-crossing edge contains at least two vertices from the same part, it cannot be contained in any transversal of $\mathcal P$.
Deleting all non-crossing edges therefore preserves the independent transversals and does not increase any vertex degree or any set codegree.
Thus, in proving Lemma~\ref{lem:indep-trans}, we may assume that every edge is crossing.

We next derive from Lemma~\ref{lem:DP} the following sufficient condition for the existence of an independent transversal.

\begin{lemma}\label{lem:indep-trans}
For every integer $r\geq3$ and every $\eta\in(0,1)$, there exist constants $c, q_0>0$ such that the following holds for every $q\geq q_0$.
Let $\cH$ be an $r$-graph with a vertex partition $V(\cH)=V_1\cup\cdots\cup V_m$ satisfying
\begin{itemize}
    \item $|V_i|\geq q$ for every $i\in[m]$;
    \item $\Delta_1(\cH)\leq cq^{r-1}\log q$;
    \item $\Delta_j(\cH)\leq q^{r-j-\eta}$ for every $2\leq j<r$.
\end{itemize}
Then $\cH$ has an independent transversal.
\end{lemma}

\begin{proof}

Set $\beta:=\eta/2$, and let $D_\beta\geq2$ and $\alpha>0$ be the constants given by Lemma~\ref{lem:DP}.
Choose $c:=\alpha/3^r$ and an integer $q_0$ such that
$q_0\geq2(D_\beta+1)$ and $3^{r-2-\beta}(q_0)^{-\beta}\leq1$.
For $q\geq q_0$, put $D:=\lfloor q/2\rfloor$.
Then $D\geq D_\beta+1\geq3$, $2D\leq q<3D$, and $\log q\leq3\log D$.

Let $A:=\{a_1,\ldots,a_m\}$ be a set disjoint from $V(\cH)$, and let $B:=V(\cH)$.
Construct a bipartite graph $G=(A,B)$ by connecting $a_i$ to every vertex of $V_i$ for each $i\in[m]$.
Define the bijection $\phi:V(\cH)\to E(G)$ by $\phi(v):=a_iv$ for $v\in V_i$, and write $\phi(S):=\{\phi(v):v\in S\}$ for $S\subseteq V(\cH)$.
Let $\mathcal F$ be the $r$-graph with $V(\mathcal F):=E(G)$ and $E(\mathcal F):=\{\phi(e):e\in E(\cH)\}$.
Since every $e\in E(\cH)$ is crossing, $\phi(e)$ is a matching of size $r$ in $G$.
Thus $\mathcal F$ is a configuration hypergraph for $G$.
For every $j\in[r-1]$ and every $S\in\binom{V(\cH)}{j}$, the edges of $\mathcal F$ containing $\phi(S)$ correspond exactly to the edges of $\cH$ containing $S$.
Hence $d_{\mathcal F}(\phi(S))=d_{\cH}(S)$, and the bijectivity of $\phi$ gives $\Delta_j(\mathcal F)=\Delta_j(\cH)$ for every $j\in[r-1]$.

We now verify the hypotheses of Lemma~\ref{lem:DP}.
Since $D\geq D_\beta\geq2$, every $a_i\in A$ satisfies
\[
d_G(a_i)=|V_i|\geq q\geq2D\geq(1+D^{-\alpha})D.
\]
Moreover, $d_G(v)=1\leq D$ for every $v\in B$.
Using $q\leq3D$, $\log q\leq3\log D$, and the choice of $c$, we obtain
\[
\Delta_1(\mathcal F)=\Delta_1(\cH)\leq cq^{r-1}\log q\leq c3^rD^{r-1}\log D=\alpha D^{r-1}\log D.
\]
For $2\leq j<r$, put $s:=r-j\in[r-2]$.
Since $D\geq q/3$ and $\eta=2\beta$, the choice of $q_0$ gives
\[
\frac{\Delta_j(\mathcal F)}{D^{s-\beta}}\leq\frac{q^{s-\eta}}{D^{s-\beta}}\leq3^{s-\beta}q^{-\beta}\leq3^{r-2-\beta}(q_0)^{-\beta}\leq1.
\]
Thus $\Delta_j(\mathcal F)\leq D^{r-j-\beta}$ for every $2\leq j<r$.

All hypotheses of Lemma~\ref{lem:DP} are therefore satisfied, so $G$ contains an $\mathcal F$-avoiding $A$-perfect matching $M$.
Its endpoints in $B$ form a set $T$ containing exactly one vertex from each part $V_i$.
If an edge $e\in E(\cH)$ were contained in $T$, then $\phi(e)\subseteq M$ would be an edge of $\mathcal F$, contradicting the fact that $M$ is $\mathcal F$-avoiding.
Therefore, $T$ is an independent transversal of $\cH$.
\end{proof}

\section{Proof of Theorem~\ref{thm:mainthm}}\label{sec:pf-main}
Our first step is to show that $\cH$ admits an independent transversal. We then prove that it contains at least $\lceil(1-\varepsilon)b\rceil$ pairwise vertex-disjoint independent transversals.

We first consider the special case in which all parts have the same size and the maximum codegrees are bounded.

\begin{claim}\label{claim:max-degree}
For every integer $r\geq3$ and every $k>0$, there exist constants
$\gamma>0$ and $q_0>0$ such that the following holds.
Let $\cH$ be an $r$-graph with a partition
$V(\cH)=V_1\cup\cdots\cup V_m$ satisfying
$|V_1|=\cdots=|V_m|=q\geq q_0$,
$\Delta_1(\cH)\leq kq^{r-1}$, and
$\Delta_j(\cH)\leq\gamma q^{r-j}$ for every $2\leq j<r$.
Then, the hypergraph $\cH$ has an independent transversal.
\end{claim}

\begin{proof}
Let $\lambda:=\max\{1,k\}+2$.
Apply Lemma~\ref{lem:halving} with $\alpha=\lambda$ to obtain constants $\ell,z,q_{\ref{lem:halving}}>0$.
Also, apply Lemma~\ref{lem:indep-trans} with $\eta=1/4$ to obtain constants $c,q_{\ref{lem:indep-trans}}>0$.

Recall that $\sigma(x)=\sqrt{\log x/x}$.
Choose a constant $C>0$ such that
$\sum_{h=0}^{\infty}2^{-h/2}\sqrt{1+h}<C$.
Choose a sufficiently large integer $q_0=q_0(r,k)$ such that
$q_0\geq\max\{q_{\ref{lem:halving}},q_{\ref{lem:indep-trans}},e\}$, $k+1\leq c\log q_0$,
$Cz\sigma(q_0)\leq\frac12(3q_0)^{-1/4}$ and $\ell\sigma(q_0)+{1}/{q_0}\leq1/6$.
Such a choice is possible since
$\sigma(x)=o(x^{-1/4})$ as $x\to\infty$.

Set
$\gamma:=\frac12(3q_0)^{-1/4}$.
Starting with $\cH_1:=\cH$ and $q_1:=q$, we construct a sequence
of induced subhypergraphs.
At each $i$-th step, the parts of $\cH_i$ have same size $q_i\geq q_0$.
Then we have
\begin{equation}\label{eq:max-degree}
\frac{\Delta_1(\cH_i)}{q_i^{r-1}}
\leq k+z\sum_{\ell=1}^{i-1}\sigma(q_\ell)
\end{equation}
and for every $2\leq j<r$,
\begin{equation}\label{eq:higher-codegree}
\frac{\Delta_j(\cH_i)}{q_i^{r-j}}
\leq \gamma+z\sum_{\ell=1}^{i-1}\sigma(q_\ell).
\end{equation}
Note that both \eqref{eq:max-degree} and \eqref{eq:higher-codegree} hold when $i=1$.

Suppose that $\cH_1,\ldots,\cH_i$ have been constructed and that
$q_{\ell+1}\leq q_\ell/2$ for each $\ell\in [i-1]$ and thus $q_\ell\geq 2q_{\ell-1}\ge \cdots\ge  2^{i-\ell}q_i$.
Since $\sigma(x)=\sqrt{\log x/x}$ is decreasing on $[e,\infty)$, 
it follows that
\begin{eqnarray}
 \sum_{\ell=1}^{i-1}\sigma(q_\ell)
\leq
\sum_{h=0}^{i-1}\sigma(2^h q_{i})
&=&
\sigma(q_{i})\sum_{h=0}^{i-1}2^{-h/2}\sqrt{1+\frac{h \log 2}{\log q_{i}}} \nonumber\\
&\leq&
\sigma(q_{i})\sum_{h=0}^{i-1}2^{-h/2}\sqrt{1+h}
<
C\sigma(q_{i})
\leq
C\sigma(q_0).\label{eq:error-sum}
\end{eqnarray}
Combining~\eqref{eq:max-degree}--\eqref{eq:error-sum} gives
\[
\frac{\Delta_1(\cH_i)}{q_i^{r-1}}
\leq k+Cz\sigma(q_0)
\leq k+\gamma
<k+1<\lambda
\]
and, for every $2\leq j<r$,
\[
\frac{\Delta_j(\cH_i)}{q_i^{r-j}}
\leq \gamma+Cz\sigma(q_0)
\leq 2\gamma
=(3q_0)^{-1/4}<1<\lambda.
\]
So the hypotheses of Lemma~\ref{lem:halving} hold at every step.

If $q_i\leq3q_0$, this process stops and set $t:=i$.
Otherwise, since $q_i\geq q_0\geq q_{\ref{lem:halving}}$,
the preceding bounds allow us to apply Lemma~\ref{lem:halving} to $\cH_i$.
This yields an induced subhypergraph $\cH_{i+1}$ whose parts have same size
$q_{i+1}:=\left\lfloor{q_i}/{2}
-\ell\sqrt{q_i\log q_i}\right\rfloor$.
So we have
$q_{i+1}
\geq {q_i}/{2}-\ell\sqrt{q_i\log q_i}-1
\geq {q_i}/{3}>q_0$.
Moreover, the maximum codegree estimate~\eqref{equ:norm-degree} in
Lemma~\ref{lem:halving} implies
\eqref{eq:max-degree} and~\eqref{eq:higher-codegree}
with $i+1$ in place of $i$.
Thus this construction can be continued whenever $q_i>3q_0$.

Since $q_{i+1}\leq q_i/2$ at each step, the process terminates at $\cH_t$ for some $t\geq1$, with $q_0\leq q_t\leq3q_0$.
The preceding estimates give
\[
\frac{\Delta_1(\cH_t)}{q_t^{r-1}}
<k+1\leq c\log q_0\leq c\log q_t
\]
and, for every $2\leq j<r$,
\[
\frac{\Delta_j(\cH_t)}{q_t^{r-j}}
\leq 2\gamma
=(3q_0)^{-1/4}
\leq q_t^{-1/4}.
\]
So $\Delta_1(\cH_t)\leq cq_t^{r-1}\log q_t$ and
$\Delta_j(\cH_t)\leq q_t^{r-j-1/4}$ for every $2\leq j<r$.

Since $q_t\geq q_0\geq q_{\ref{lem:indep-trans}}$,
Lemma~\ref{lem:indep-trans} with $\eta=1/4$ yields an independent
transversal of $\cH_t$.
As $\cH_t$ is an induced subhypergraph of $\cH$, this is also an independent transversal of $\cH$.
\end{proof}

Now we are ready to prove Theorem~\ref{thm:mainthm}.

\begin{proof}[Proof of Theorem~\ref{thm:mainthm}]
Fix an integer $r\geq3$, a constant $k>0$, and $\varepsilon\in(0,1)$.
Set $K=k(2/\varepsilon)^r$, and let $\delta>0$ and $q_0\geq1$ be the constants given by Claim~\ref{claim:max-degree} with $K$.
Choose
$\gamma=\delta\left({\varepsilon}/{2}\right)^{r-2}$ and $b_0=\left\lceil{2q_0}/{\varepsilon}\right\rceil$.

Let $\cH$ be an $r$-graph satisfying the hypotheses of the theorem with $b\geq b_0$.
For each $i\in[m]$, let
$W_i=\left\{v\in V_i:d_{\cH}(v)\leq({2k}/{\varepsilon})b^{r-1}\right\}$.
Then it gives
\[
|V_i\setminus W_i|\leq\frac{\varepsilon}{2kb^{r-1}}\sum_{v\in V_i}d_{\cH}(v)\leq\frac{\varepsilon}{2}|V_i|.
\]
So we get
$|W_i|\geq\left(1-\frac{\varepsilon}{2}\right)|V_i|\geq\left(1-\frac{\varepsilon}{2}\right)b$.

Set $N=\lceil(1-\varepsilon)b\rceil$ and $q=\lceil\varepsilon b/2\rceil$, so that $q\geq q_0$.
We construct $N$ pairwise vertex-disjoint independent transversals using only vertices in $\bigcup_{i=1}^m W_i$.

Suppose that $s<N$ such transversals $T_1,\ldots,T_s$ have already been chosen.
For each $i\in[m]$, let
$W_i^{(s)}=W_i\setminus\bigcup_{t=1}^s T_t$.
Since $s<(1-\varepsilon)b$, each previously chosen transversal contains exactly one vertex from $W_i$, and hence
$\left|W_i^{(s)}\right|=|W_i|-s\geq\left(1-\frac{\varepsilon}{2}\right)b-s>{\varepsilon b}/{2}$.
Thus, for each $i\in[m]$, there is a subset $U_i\subseteq W_i^{(s)}$ of size $q$.

Let $\cH'=\cH[\bigcup_{i=1}^m U_i]$.
Since for any $v\in V(\cH')$, one knows that $v\in W_i$ for some $i\in [m]$.
Note that $b\leq2q/\varepsilon$.
Therefore,
\[
\Delta_1(\cH')\leq\frac{2k}{\varepsilon}b^{r-1}\leq\frac{2k}{\varepsilon}\left(\frac{2q}{\varepsilon}\right)^{r-1}=Kq^{r-1}.
\]
Moreover, for every $2\leq j<r$,
\[
\Delta_j(\cH')\leq\Delta_j(\cH)\leq\gamma b^{r-j}\leq\gamma\left(\frac{2}{\varepsilon}\right)^{r-j}q^{r-j}=\delta\left(\frac{\varepsilon}{2}\right)^{j-2}q^{r-j}\leq\delta q^{r-j}.
\]
Since all $U_i$ have size $q\geq q_0$, Claim~\ref{claim:max-degree} gives an independent transversal $T_{s+1}$ of $\cH'$ with respect to $\{U_1,\ldots,U_m\}$.
As $\cH'$ is induced and $U_i\subseteq V_i$ for every $i\in[m]$, this is also an independent transversal of $\cH$.
It is disjoint from $T_1,\ldots,T_s$ by construction.

This construction process can therefore be continued until $N=\lceil(1-\varepsilon)b\rceil$ pairwise vertex-disjoint independent transversals have been found.
\end{proof}

\bigskip
\noindent{\bf \large Declaration on the use of AI} \\
The proof was developed with the help of ChatGPT Pro model in which the
authors guided its progress. The authors take full responsibility for the correctness of the paper.
\bibliographystyle{abbrv}
\bibliography{refs}
\end{document}